\documentclass[a4paper,12pt,reqno]{amsart}

\usepackage{amsfonts,amssymb}
\usepackage[alphabetic,lite]{amsrefs}

\usepackage[cmtip,arrow]{xy}
\usepackage{pb-diagram,pb-xy}

\usepackage{etoolbox}
\apptocmd{\sloppy}{\hbadness 10000\relax}{}{}

\newtheorem{theorem}{Theorem}[section]
\newtheorem{lemma}[theorem]{Lemma}
\newtheorem{proposition}[theorem]{Proposition}
\newtheorem{cor}[theorem]{Corollary}
\theoremstyle{definition}
\newtheorem{definition}[theorem]{Definition}

\theoremstyle{remark}
\newtheorem{remark}[theorem]{Remark}

\newcommand{\NN}{\mathbb{N}} 
\newcommand{\CC}{\mathbb{C}} 
\newcommand{\mat}[3]{{\mathrm{Mat}\left( #1 \times #2 ; \, #3 \right)}} 
\newcommand{\udisk}{\mathbb{D}} 
\newcommand{\sdisk}{\mathbb{G}} 
\newcommand{\slgrp}{\mathrm{SL}} 
\newcommand{\slalg}{\mathfrak{sl}} 
\newcommand{\sball}{\Omega} 
\newcommand{\hol}{\mathcal{O}} 

\newcommand{\branch}{\mathop{\mathrm{br}}}

\title{Lifting to the spectral ball with interpolation}
\author{Rafael B. Andrist}
\address{Rafael B. Andrist \\ Fachbereich C -- Mathematik \\ Bergische Universit\"at Wuppertal \\ Germany}
\email{rafael.andrist@math.uni-wuppertal.de}

\thanks{Acknowledgements: The author would like to thank Frank Kutzschebauch (University of Bern) and Franc Forstneri\v{c} (University of Ljubljana) for interesting and helpful discussions.}

\subjclass[2010]{30E05 (primary), and 32Q55, 32M10 (secondary)}

\keywords{spectral ball, Nevanlinna-Pick, symmetrized polydisc, Oka principle, interpolation}
 
\begin{document}

\begin{abstract}
We give necessary and sufficient conditions for solving the spectral Nevanlinna--Pick lifting problem. This reduces the spectral Nevanlinna--Pick problem to a jet interpolation problem into the symmetrized polydisc.
\end{abstract}

\maketitle

\section{Introduction}

The spectral ball is the set of square matrices with spectral radius less than $1$. It appears naturally in Control Theory \cites{robustcontrol-proc, robustcontrol}, but is also of theoretical interest in Several Complex Variables.

\begin{definition}
The \emph{spectral ball} of dimension $n \in \NN$ is defined to be
\[
\sball_n := \{ A \in \mat{n}{n}{\CC} \;:\; \rho(A) < 1 \}
\]
where $\rho$ denotes the spectral radius, i.e.\ the modulus of the largest eigenvalue.
\end{definition}
Note that in dimension $n=1$, the spectral ball is just the unit disc. We will assume throughout the paper that $n \geq 2$.

\smallskip

The Nevanlinna--Pick problem is an interpolation problem for holomorphic functions on the unit disc $\udisk$. The classical \emph{Nevanlinna--Pick problem} for holomorphic functions $\udisk \to \udisk$ with interpolation in a finite set of points has been solved by Pick \cite{Pick} and Nevanlinna \cite{Nevanlinna}. The \emph{spectral Nevanlinna--Pick problem} is the analogue interpolation problem for holomorphic maps $\udisk \to \sball_n$:
\begin{quote}
Given $m \in \NN$ distinct points $a_1, \dots, a_m \in \udisk$, decide whether there is a holomorphic map $F \colon \udisk \to \sball_n$ such that 
\[F(a_j) = A_j, \; j = 1, \dots, m\]
for given matrices $A_1, \dots, A_m \in \sball_n$.
\end{quote}
It has been studied by many authors, in particular by Agler and Young for dimension $n=2$ and generic interpolation points \cites{AglerYoung2000, AglerYoung2004}. Bercovici \cite{Bercovici} has found solutions for $n=2$, and Costara \cites{Costara2} has found solutions for generic interpolation points in higher dimensions. In general, the problem is still open.

The spectral ball $\sball_n$ can also be understood in the following way: Denote by $\sigma_1, \dots, \sigma_n \colon \CC^n \to \CC$ the elementary symmetric polynomials in $n$ complex variables. Let $ \mathrm{EV} \colon \mat{n}{n}{\CC} \to \CC^n$ assign to each matrix a vector of its eigenvalues.
Then we denote by $\pi_1 := \sigma_1 \circ \mathrm{EV}, \dots, \pi_n := \sigma_n \circ \mathrm{EV}$ the elementary symmetric polynomials in the eigenvalues. By symmetrizing we avoid any ambiguities of the order of eigenvalues and obtain a polynomial map $\pi = (\pi_1, \dots, \pi_n)$, symmetric in the entries of matrices in $\mat{n}{n}{\CC}$, actually
\[
\chi_A(\lambda) = \lambda^n + \sum_{j=1}^n (-1)^j \cdot \pi_j(A) \cdot \lambda^{n-j}
\]
where $\chi_A$ denotes the characteristic polynomial of $A$.

Now we can consider the holomorphic surjection $\pi \colon \Omega_n \to \sdisk_n$ of the spectral ball onto the symmetrized polydisc $\sdisk_n := (\sigma_1, \dots, \sigma_n)(\udisk^n)$. A generic fibre, i.e.\ a fibre above a base point with no multiple eigenvalues, consists exactly of one equivalence class of similar matrices. Thus, a generic fibre is actually a $\slgrp_n(\CC)$-homogeneous manifold where the group $\slgrp_n(\CC)$ acts by conjugation. A singular fibre decomposes into several strata which are $\slgrp_n(\CC)$-homogeneous manifolds as well, but not necessarily connected.

Given this holomorphic surjection, it is natural to consider a weaker version of the spectral Nevanlinna--Pick problem, which is called the \emph{spectral {Nevan\-linna}\-{--}Pick lifting problem}:
\begin{quote}
Given $m \in \NN$ distinct points $a_1, \dots, a_m \in \udisk$, and a holomorphic map $f \colon \udisk \to \sdisk_n$ with $f(a_j) = \pi(A_j)$ for given matrices $A_1, \dots, A_m \in \sball_n$,
decide whether there is a holomorphic map $F \colon \udisk \to \sball_n$ such that \[F(a_j) = A_j, \; j = 1, \dots, m.\]
i.e.\ such that the following diagram commutes:
\begin{center}
$
\begin{diagram}
\node[2]{\sball_n} \arrow{s,r}{\pi} \\
\node{\udisk} \arrow{ne,t,--}{F} \arrow{e,t}{f} \node{\sdisk_n} \\
\end{diagram}
\vspace*{-12mm}
$
\\
$
a_1, \dots, a_m \mapsto \pi(A_1), \dots, \pi(A_m)
$
\end{center}
\end{quote}

When this lifting problem is solved, the spectral Nevanlinna--Pick problem reduces to an interpolation problem $\udisk \to \sdisk_n$. In contrast to the spectral ball, the symmetrized polydisc $\sdisk_n$ is a taut domain, and should be more accessible with techniques from hyperbolic geometry.
Solutions to this lifting problem have been found for dimensions $n = 2, 3$ by Nikolov, Pflug and Thomas \cite{TwoThree} and recently for dimensions $n = 4, 5$ by  Nikolov, Thomas and Tran \cite{locallift}. They also provide the solution to a localised version of the spectral Nevanlinna--Pick lifting problem:

\begin{proposition}[\cite{locallift}*{Proposition 7}]
\label{locallifting}
Let $f \in \hol(U, \sdisk^n)$, where $U$ is a neighborhood of $p \in \udisk$, and let $M \in \sball_n$.
Then there exists a neighborhood $U^\prime \subset U$ of $p$ and
$F \in \hol(U^\prime, \sball_n)$ such that
$\pi \circ F = f$, $F(p) = M$ and $F(v)$ is cyclic for $v \in U^\prime \setminus \{p\}$ if and only if
\begin{equation}
\label{localcondition}
\tag{$\ast$}
\left.\frac{d^k}{d \lambda^k} {{\mbox{\Large$\chi$}}}_{f(v)}(\lambda)\right|_{\lambda = \lambda_j} = O((v - p)^{d_{n_j-k}(B_j)}), \quad 0 \leq k \leq n_j - 1, \; 1 \leq j \leq s
\end{equation}
where $B_1 \in \mat{n_1}{n_1}{\CC}, \dots, B_s \in \mat{n_s}{n_s}{\CC}$ are the Jordan blocks of $M \simeq B_1 \oplus \dots \oplus B_s$ and 
\[
\begin{split}
d_\ell(B_j) = \min\Big\{ d \in \NN \,:\; \exists x_1, \dots x_d \in \CC^n\, \text{s.t.} \\ \dim \mathop{\mathrm{span}_\CC}\left\{ B^{k_1} x_1, \dots B^{k_d} x_d; \, k_1, \dots, k_d \geq 0\right\} \geq \ell \Big\}
\end{split}
\]
\end{proposition}

A matrix $M \in \mat{n}{n}{\CC}$ is called \emph{cyclic}, if it admits a cyclic vector $x \in \CC^n$, i.e.\ $\mathop{\mathrm{span}_\CC} \left\{ M^k \cdot x, \; k \in \NN\right\} = \CC^n$. This is equivalent to saying that there is only one Jordan block per eigenvalue in the Jordan block decomposition of $M$.

\begin{remark}
It is easy to see that the condition \eqref{localcondition} is necessary: it is a direct consequence of the factorisation $\pi \circ F = f$ on the jet level, explicitly stated using the Jordan block decomposition.
\end{remark}

In this paper we prove that the solution of the localised spectral Nevanlinna--Pick lifting problem implies the global one:

\begin{theorem}
\label{mainthm}
The spectral Nevanlinna--Pick lifting problem can be solved if and only if it can be solved locally around the interpolation points which means that \eqref{localcondition} holds for each interpolation point.
\end{theorem}

As a by-product of the proof we obtain (see Corollary \ref{noSPSH}) that the spectral ball admits no bounded from above strictly plurisubharmonic functions, but non-constant bounded holomorphic functions.

\pagebreak

\section{Oka Theory}

In this section we provide some background in Oka theory that will be needed in the proof of the main theorem. The goal is to find conditions for the existence of holomorphic liftings or holomorphic sections. We first recall the following equivalence:

\begin{remark}
\label{pullbacklifting}
For a holomorphic surjection $\pi \colon Z \to X$ and a holo\-morphic map $f \colon Y \to X$, we denote the pull-back of $Z$ under $f$ by
\begin{equation*}
f^\ast(Z) = \left\{ (y, z) \in Y \times Z \;:\; f(y) = \pi(z) \right\}
\end{equation*}
Together with the natural projections to the first resp. second factor $f^\ast(\pi) \colon f^\ast(Z) \to Y$ resp. $f^\ast(Z) \to Z$ we obtain the following commutative diagram
\[
\begin{diagram}
\node{f^\ast(Z)} \arrow[2]{e,t}{} \arrow{s,l}{f^\ast(\pi)} \node[2]{Z} \arrow{s,r}{\pi} \\
\node{Y} \arrow{e,=} \node{Y} \arrow{e,t}{f} \arrow{ne,t,--}{F} \arrow{nw,t,--}{f^\ast(F)} \node{X}
\end{diagram}
\]
We see that the existence of the lifting $F$ is equivalent to the existence of the lifting $f^\ast(F), \, y \mapsto (y, F(y))$. Thus, the lifting problem can always be reduced to finding a section of $f^\ast(Z) \to Y$.
\end{remark}

The crucial ingredient will be an Oka theorem for branched holomorphic maps due to Forstneri\v{c}.
%
%
From \cite{Okamulti} we recall the following definitions:
\begin{definition}
A point $z \in Z$ is a \emph{branching point} of a holomorphic surjection between complex spaces $\pi \colon Z \to X$ if $\pi$ is not submersive in $z$. The \emph{branching locus} of $\pi$ is denoted by $\mathop{\mathrm{br}} \pi$ and consists of all branching points. For maps between complex spaces, also $Z_{\mathrm{sing}}$ and $\pi^{-1}\left(X_{\mathrm{sing}}\right)$ are included in $\mathop{\mathrm{br}} \pi$ by convention.
\end{definition}

\begin{definition}
A holomorphic surjection $\pi \colon Z \to X$ between complex spaces is called an \emph{elliptic submersion} over an open subset $V \subseteq X_{\mathrm{reg}}$ if
\nopagebreak
\begin{enumerate}
\item $\pi|\pi^{-1}(V)$ is a submersion of complex manifolds
\item each point $x \in V$ has an open neighborhood $U \subseteq V$ such that there exists a holomorphic vector bundle $E \to \pi^{-1}(U)$ together with a so-called holomorphic \emph{dominating spray} $s \colon E \to \pi^{-1}(U)$ satisfying the following conditions for each $z \in \pi^{-1}(U)$:
\begin{enumerate}
\item $s(E_z) \subseteq Z_{\pi(z)}$
\item $s(O_z) = z$
\item the derivative $d s \colon T_{0_z} E \to T_z Z$ maps the subspace $E_z \subseteq T_{0_z} E$ surjectively onto the vertical tangent space $\mathop{\mathrm{ker}} d_z \pi$.
\end{enumerate}
\end{enumerate}
\end{definition}

A examination of the proof in \cite{Okamulti} shows that a small variation of the statement of \cite{Okamulti}*{Theorem 2.1} is valid with exactly the same proof:
\begin{theorem}
\label{Okabranched}
Let $X$ be a Stein space and $\pi \colon Z \to X$ be a holomorphic map of a complex space $Z$ onto $X$. Assume that $Z \setminus \branch{\pi}$ is an elliptic submersion over $X$.
Let $X_0$ be a (possibly empty) closed subvariety of $X$.
Given a continuous section $F \colon X \to Z$ which is holomorphic in an open set containing $X_0$ and such that $F(X \setminus X_0) \subseteq Z \setminus \branch{\pi}$,
there is for any $k \in \NN$ a homotopy of continuous sections $F_t \colon X \to Z, \; t \in [0, 1],$ such that $F_0 = F$, for each $t \in [0, 1]$ the section $F_t$ is holomorphic in a neighborhood of $X_0$ and tangent to order $k$ along $X_0$, and $F_1$ is holomorphic on $X$.
If $F$ is holomorphic in a neighborhood of $K \cup X_0$ for some compact, holomorphically convex subset $K$ of $X$ then we can choose $F_t$ to be holomorphic in a neighborhood of $K \cup X_0$ and to approximate $F = F_0$ uniformly on $K$.
\end{theorem}

\begin{remark}
The difference to \cite{Okamulti}*{Theorem 2.1} is only that we allow $X_0$ to be smaller than $\pi (\branch \pi)$ provided that the section over $X_0 \setminus \pi (\branch \pi)$ does not hit the branching locus $\branch \pi$.
\end{remark}

%

\section{Proof}

In case of the spectral Nevanlinna--Pick lifting problem, we have $X = \sdisk_n, Y = \udisk, Z = \sball_n$, and $X_0 := \{a_1, \dots, a_m \}$ is a closed subvariety in $\udisk$. We note that $Y$ is indeed Stein and that $\pi$ is an elliptic submersion in the cyclic matrices since the largest strata of all fibres are $\slgrp_n(\CC)$-homogeneous spaces of dimension $n$, see for example \cite{Okabook}*{Example 5.5.13}.
We write down the spray explicitly for any $U \subseteq \sdisk_n$: we choose the bundle $E := \slalg_n(\CC) \times \pi^{-1}(U) \to \pi^{-1}(U)$ with the natural projection, and define the spray $s \colon E \to \pi^{-1}(U)$ as
\[
(B, A) \mapsto \exp(B) \cdot A \cdot \exp(-B)
\]
The derivative in $(0, A)$, evaluated for a tangent vector $C$, is then given by the adjoint representation of $\slalg_n(\CC)$:
\[
d_{(0, A)} s (C) = [A, C]
\]
and the conditions of the elliptic submersion are easily verified.

Now, by remark \ref{pullbacklifting} and the fact that the pullback of an elliptic submersion is still an elliptic submersion (see \cite{Okamulti}*{Lemma 3.1}) we translate the problem of finding a lifting $F \colon X \to Z$ with $\pi \circ F = f$ into a problem of finding a section of $f^{\ast}(\pi) \colon f^{\ast}(\sball_n) \to \udisk$. 

\medskip
Due to a lack of reference, we also include the following lemma and its proof:
\begin{lemma}
\label{connected}
Each fibre, the largest stratum of each fibre and also each connected component of any stratum of each fibre of $\pi \colon \sball_n \to \sdisk_n$ is $\CC$-connected.
\end{lemma}
\begin{proof}
Given any two similar matrices $B$ and $C$, we can connect them through finitely many holomorphic images of complex lines:
there exists a finite number of matrices $Q_1, \dots, Q_s \in \slalg_n(\CC)$ such that
\[
C = \exp(Q_s) \cdots \exp(Q_1) \cdot B \cdot \exp(-Q_1) \cdots \exp(-Q_s).
\]
By appending a complex time parameter $t_s$ in front of each $Q_s$ we obtain $s$ lines connecting $B$ to $C$. This proves that the largest stratum (hence also a generic fibre) or any connected component of a stratum is $\CC$-connected.
For the connectedness of a singular fibre it is now sufficient to connect the strata. This can be done by introducing parameters in the off-diagonal entries of the Jordan-block decomposition.
\end{proof}

We obtain the following corollary which will however not play any role in the proof of the main theorem:

\begin{cor}
\label{noSPSH}
The spectral ball $\sball_n, n \geq 2,$ is a union of immersed complex lines. Hence there exists no bounded from above strictly pluri\-subharmonic function on $\sball_n$.
\end{cor}
\begin{proof}
The pull-back of any (strictly) plurisubharmonic function to any of these complex lines is (strictly) subharmonic, and Liouville's Theorem applies.
\end{proof}

\bigskip

Now we are ready to prove Theorem \ref{mainthm}:

First we need understand the branching locus of the pull-back of $\pi \colon \sball_n \to \udisk$ by $f$, $f^\ast(\pi) \colon f^\ast(\sball_n) \to \udisk$.
Let $(z, A) \in f^\ast(\sball_n) \subset \udisk \times \sball_n \subset \CC \times \CC^{n^2}$. Assume that $(z, A) \in f^\ast(\sball_n)$ is a singularity. This means that the derivative of the defining equations of $f^\ast(\sball_n)$ has not full rank in $(z, A)$,
\begin{equation*}
\exists v \in \mat{n}{1}{\CC}, v \neq 0, \;:\; \left( f'(z), -d_A \pi \right)^\intercal \cdot v = 0
\end{equation*}
But then also $(d_A \pi)^\intercal \cdot v = 0$ and necessarily $A \in \branch \pi$.
Assume that $(z, A) \in f^\ast(\sball_n)$ is a branching point of $f^\ast(\pi)$ which projects $(z, A)$ to $z$, i.e.\ for $d_{(z, A)} f^\ast(\pi) = (1, 0, \dots, 0)^t$ it holds outside singularities that 
\begin{equation*}
\exists w \in \mat{n}{1}{\CC}, w \neq 0, \;:\; \left( f'(z), -d_A \pi \right)^\intercal \cdot w = (1,0, \dots, 0)^\intercal
\end{equation*}
But then again $(d_A \pi)^\intercal \cdot w = 0$ and necessarily $A \in \branch \pi$.
Hence, $\branch f^\ast(\pi)$ which by definition includes also the singularities of $f^\ast(\sball_n)$, is contained in the pull-back of $\branch \pi$, and no new branching points are introduced.

\smallskip
Theorem \ref{Okabranched} gives the desired section $f^{\ast}(F) = g_1$, once the existence of $g_0$ is clear. Proposition \ref{locallifting} provides the necessary local liftings in each point $a_1, \dots, a_m$ which are pulled back by $f$ to local holomorphic sections such that they do not hit $\branch{f^\ast(\pi)}$ except for possibly the interpolation point itself.
Choosing the neighborhoods for the local sections small enough and contractible, we achieve that $g_0$ is holomorphic in a neighborhood $U$ of $X_0$.

\smallskip
It remains to show that we can extend $g_0$ as continuous section over $\udisk \setminus U$. Since all the components of $U$ are contractible, this is equivalent to find a continuous section $g_0$ which interpolates the points.
We can connect the finitely many points $a_1, \dots, a_m$ by arcs such that the union $\gamma$ of these arcs is homeomorphic to an interval, and then contract the unit disc to $\gamma$. Since all the fibres are connected (even $\CC$-connected by Lemma \ref{connected}) there is no topological obstruction to construct a continuous interpolating section $\gamma \to \sball_n$.

We have proved Theorem \ref{mainthm}.

\end{document}